\documentclass[proceedings%
]{aofa}

\usepackage[utf8]{inputenc}
\usepackage{subfigure}

\usepackage[round]{natbib}

\usepackage{amsmath}
\usepackage{amssymb}
\usepackage{graphicx}
\usepackage{mathrsfs}  
\usepackage{amssymb}

\def\beq*{\begin{eqnarray*}}
\def\eeq*{\end{eqnarray*}}

\def\be{\begin{equation}}
\def\ee{\end{equation}}

\def\mcP{\mathcal{P}}
\def\mcT{\mathcal{T}}
\def\mcB{\mathcal{B}}

\def\mbP{\mathbb{P}}
\def\mbE{\mathbb{E}}
\def\mcF{\mathcal{F}}
\def\mcX{\mathcal{X}}
\def\Bin{\operatorname{Bin}}
\newtheorem{theorem}{Theorem}
\newtheorem{lemma}[theorem]{Lemma}
\newtheorem{corollary}[theorem]{Corollary}
\newtheorem{proposition}[theorem]{Proposition}
\newtheorem{definition}{Definition}

\numberwithin{subcase}{case}

\author[P. Hitczenko and A. Lohss]{Pawe{\l} Hitczenko\addressmark{1}\thanks{Partially supported by a grant from 
Simons Foundation (grant \#208766)} and Amanda Lohss\addressmark{1}}

\title[Corners in tree--like tableaux]{Corners in tree--like tableaux}

\address{\addressmark{1}Department of Mathematics, Drexel University, Philadelphia, 
PA  19104, USA}

\begin{document}
\maketitle
\begin{abstract}
\paragraph{Abstract.}
 In this paper, we study tree--like tableaux, combinatorial objects which exhibit a natural tree structure and are connected to the partially asymmetric simple exclusion process (PASEP). There was a conjecture made on the total number of corners in tree--like tableaux and the total number of corners in symmetric tree--like tableaux. In this paper, we prove the first conjecture leaving the proof of the second conjecture to the full version of this paper. Our proofs are based on the bijection with permutation tableaux or type--B permutation tableaux and consequently, we also prove results for these tableaux.  

 \end{abstract}

\section{Introduction}
Tree--like tableaux are relatively new objects which were introduced in \cite{ABN}. They are in bijection with permutation tableaux and alternative tableaux but are interesting in their own right as they exhibit a natural tree structure (see \cite{ABN}). They also provide another avenue in which to study the partially asymmetric simple exclusion process (PASEP), an important model from statistical mechanics. See \cite{ABN} and \cite{LZ} for more details on the connection between tree--like tableaux and the PASEP. See also \cite{BUR}, \cite{CN},  \cite{CW4}, \cite{CW3}, \cite{PN}, \cite{SW} and \cite{XV} for more details on permutation and alternative tableaux. 

In \cite{LZ}, the expected number of occupied corners in tree--like tableaux and the number of occupied corners in symmetric tree--like tableaux were computed (see Section~\ref{sec:prel} for definitions). In addition, it was conjectured (see Conjectures 4.1 and 4.2 in \cite{LZ}) that the total number of corners in tree--like tableaux of size \begin{math}n\end{math} is \begin{math}n!\times\frac{n+4}{6}\end{math} and the total number of corners in symmetric tree--like tableaux of size \begin{math}2n+1\end{math} is \begin{math}2^n\times n!\times\frac{4n+13}{12}\end{math}.

We have proven both conjectures and in this paper, we will present the proof of the first conjecture (note that \cite{LZpers} was able to prove the first conjecture independently using a different method). The proof of the second conjecture will be given in the full version of this paper \cite{HL}. Our proofs are based on  the bijection with permutation tableaux or type--B permutation tableaux and consequently, we also have results for these tableaux (see Theorems~\ref{thm:exp-c} and~\ref{cornersTB} below for precise statements).

The rest of the paper is organized as follows. In the next section we introduce the necessary definitions and  notation. Section~\ref{seccorners} contains the proof of the conjecture for tree--like tableaux. Section~\ref{sec:symtab} develops the tools necessary to prove the second conjecture for symmetric tree--like tableaux. The proof then follows similarly to the proof of the first conjecture and will be left to the full version of this paper \cite{HL}.

\section{Preliminaries}\label{sec:prel}
A \emph{Ferrers diagram}, \begin{math}F\end{math}, is a left--aligned sequence of cells with weakly decreasing rows. The \emph{half--perimeter} of \begin{math}F\end{math} is the number of rows plus the number of columns. 
The \emph{border edges} of a Ferrers diagram are the edges of the southeast border, and the number of border edges is equal to the half--perimeter. We will occasionally refer to a border edge as a step (south or west). A \emph{shifted Ferrers diagram} is a diagram obtained from a Ferrers diagram with \begin{math}k\end{math} columns by adding \begin{math}k\end{math} rows above it of lengths \begin{math}k,(k-1),\dots,1\end{math}, respectively. The half--perimeter of the shifted Ferrers diagram is the same as the original Ferrers diagram (and similarly, the border edges are the same).  The right--most cells of added rows are called \emph{diagonal cells}. 

Let us recall the following two definitions introduced in \cite{ABN} and \cite{SW}, respectively.
 \begin{definition}\label{T}
A tree--like tableau of size \begin{math}n\end{math} is a Ferrers diagram of half-perimeter \begin{math}n+1\end{math} with some cells (called pointed cells) filled with a point according to the following rules:
\begin{enumerate}
\item The cell in the first column and first row is always pointed (this point is known as the root point). \label{T1}
\item Every row and every column contains at least one pointed cell. \label{T2}
\item For every pointed cell, all the cells above are empty or all the cells to the left are empty. \label{T3}
\end{enumerate}
\end{definition}

\begin{definition}\label{P}
A permutation tableau of size \begin{math}n\end{math} is a Ferrers diagram of half--perimeter \begin{math}n\end{math} filled with \begin{math}0\end{math}'s and \begin{math}1\end{math}'s according to the following rules: 
\begin{enumerate}
\item There is at least one \begin{math}1\end{math} in every column. \label{P1}
\item There is no \begin{math}0\end{math} with a \begin{math}1\end{math} above it and a \begin{math}1\end{math} to the left of it simultaneously. \label{P2} 
\end{enumerate}
\end{definition}
We will also need a notion of type--B tableaux originally introduced in \cite{LW}. Our definition follows a more explicit description 
given in \cite[Section~4]{CK}.
\begin{definition}\label{B}
A type--B permutation tableau of size \begin{math}n\end{math} is a shifted Ferrers diagram of half--perimeter \begin{math}n\end{math} filled with \begin{math}0\end{math}'s and \begin{math}1\end{math}'s according to the following rules:
\begin{enumerate}
\item There is at least one \begin{math}1\end{math} in every column. \label{B1}
\item There is no \begin{math}0\end{math} with a \begin{math}1\end{math} above it and a \begin{math}1\end{math} to the right of it simultaneously. \label{B2} 
\item If one of the diagonal cells contains a \begin{math}0\end{math} (called a diagonal \begin{math}0\end{math}), then all the cells in that row are \begin{math}0\end{math}. \label{B3}
\end{enumerate}
\end{definition}

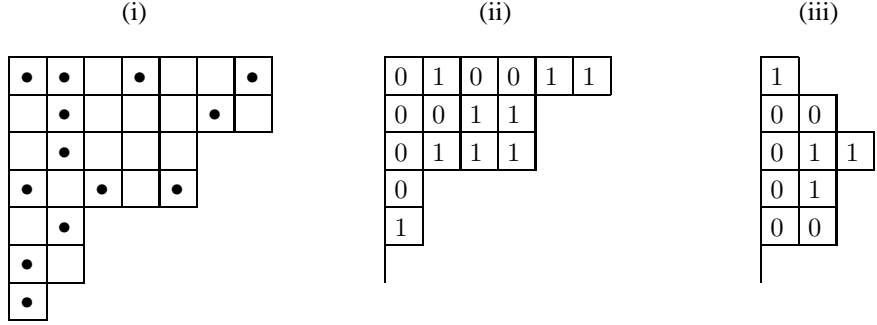
\begin{figure}[htbp]
\setlength{\unitlength}{.5cm}
\begin{center}
\begin{picture} (22,8)

\put(3, 8){(i)}
\put(12.5, 8){(ii)}
\put(21, 8){(iii)}
\put(0,0){\line(0,1){7}}
\put(1,0){\line(0,1){7}}
\put(2,1){\line(0,1){6}}
\put(3,3){\line(0,1){4}}
\put(4,3){\line(0,1){4}}
\put(5,3){\line(0,1){4}}
\put(6,5){\line(0,1){2}}
\put(7,5){\line(0,1){2}}

\put(0,7){\line(1,0){7}}
\put(0,6){\line(1,0){7}}
\put(0,5){\line(1,0){7}}
\put(0,4){\line(1,0){5}}
\put(0,3){\line(1,0){5}}
\put(0,2){\line(1,0){2}}
\put(0,1){\line(1,0){2}}
\put(0,0){\line(1,0){1}}

\put(4.3, 3.3){\begin{math}\bullet\end{math}}
\put(5.3, 5.3){\begin{math}\bullet\end{math}}
\put(6.3, 6.3){\begin{math}\bullet\end{math}}
\put(3.3, 6.3){\begin{math}\bullet\end{math}}
\put(2.3, 3.3){\begin{math}\bullet\end{math}}
\put(1.3, 2.3){\begin{math}\bullet\end{math}}
\put(1.3, 4.3){\begin{math}\bullet\end{math}}
\put(1.3, 5.3){\begin{math}\bullet\end{math}}
\put(1.3, 6.3){\begin{math}\bullet\end{math}}
\put(0.3, 3.3){\begin{math}\bullet\end{math}}
\put(0.3, 1.3){\begin{math}\bullet\end{math}}
\put(0.3, 0.3){\begin{math}\bullet\end{math}}
\put(0.3, 6.3){\begin{math}\bullet\end{math}}


\put(10,1){\line(0,1){6}}
\put(11,2){\line(0,1){5}}
\put(12,4){\line(0,1){3}}
\put(13,4){\line(0,1){3}}
\put(14,4){\line(0,1){3}}
\put(15,6){\line(0,1){1}}
\put(16,6){\line(0,1){1}}

\put(10,7){\line(1,0){6}}
\put(10,6){\line(1,0){6}}
\put(10,5){\line(1,0){4}}
\put(10,4){\line(1,0){4}}
\put(10,3){\line(1,0){1}}
\put(10,2){\line(1,0){1}}

\put(10.25, 2.25){\begin{math}1\end{math}}
\put(10.25, 3.25){\begin{math}0\end{math}}
\put(10.25, 4.25){\begin{math}0\end{math}}
\put(10.25, 5.25){\begin{math}0\end{math}}
\put(10.25, 6.25){\begin{math}0\end{math}}

\put(11.25, 6.25){\begin{math}1\end{math}}
\put(11.25, 5.25){\begin{math}0\end{math}}
\put(11.25, 4.25){\begin{math}1\end{math}}

\put(12.25, 6.25){\begin{math}0\end{math}}
\put(12.25, 5.25){\begin{math}1\end{math}}
\put(12.25, 4.25){\begin{math}1\end{math}}

\put(13.25, 6.25){\begin{math}0\end{math}}
\put(13.25, 5.25){\begin{math}1\end{math}}
\put(13.25, 4.25){\begin{math}1\end{math}}

\put(14.25, 6.25){\begin{math}1\end{math}}

\put(15.25, 6.25){\begin{math}1\end{math}}


\put(20,1){\line(0,1){6}}
\put(21,2){\line(0,1){5}}
\put(22,2){\line(0,1){4}}
\put(23,4){\line(0,1){1}}

\put(20,7){\line(1,0){1}}
\put(20,6){\line(1,0){2}}
\put(20,5){\line(1,0){3}}
\put(20,4){\line(1,0){3}}
\put(20,3){\line(1,0){2}}
\put(20,2){\line(1,0){2}}

\put(20.25, 6.25){\begin{math}1\end{math}}

\put(20.25, 5.25){\begin{math}0\end{math}}
\put(21.25, 5.25){\begin{math}0\end{math}}

\put(21.25, 4.25){\begin{math}1\end{math}}
\put(20.25, 4.25){\begin{math}0\end{math}}
\put(22.25, 4.25){\begin{math}1\end{math}}

\put(20.25, 3.25){\begin{math}0\end{math}}
\put(21.25, 3.25){\begin{math}1\end{math}}

\put(21.25, 2.25){\begin{math}0\end{math}}
\put(20.25, 2.25){\begin{math}0\end{math}}

\end{picture}
\caption{(i) A tree--like tableau of size $13$. (ii) A permutation tableau of size $12$. (iii) A type-B permutation tableau of size $6$.}
\end{center}
\end{figure}

Let \begin{math}\mathcal{T}_n\end{math} be the set of all tree--like tableaux of size \begin{math}n\end{math}, \begin{math}\mathcal{P}_n\end{math} denote the set of all permutation tableaux of size \begin{math}n\end{math}, and \begin{math}\mathcal{B}_n\end{math} denote the set of all type--B permutation tableaux of size \begin{math}n\end{math}. In addition to these tableaux, we are also interested in \emph{symmetric tree--like tableaux}, a subset of tree--like tableaux which are symmetric about their main diagonal (see \cite[Section 2.2]{ABN} for more details). 
As noticed in \cite{ABN}, the size of a symmetric tree--like tableaux must be odd, and thus, we let \begin{math}\mathcal{T}_{2n+1}^{sym}\end{math} denote the set of all symmetric tree--like tableaux of size \begin{math}2n+1\end{math}. It is a well--known fact that \begin{math}|\mathcal{P}_n|=n!\end{math} and \begin{math}|\mathcal{B}_n|=2^nn!\end{math}. Consequently, \begin{math}|\mathcal{T}_n|=n!\end{math} and \begin{math}|\mathcal{T}_{2n+1}^{sym}|=2^nn!\end{math} since by \cite{ABN}, there are bijections between these objects. 
We let  \begin{math}\mcX_n \in\{\mcT_n, \mcT^{sym}_{2n+1}, \mcP_n,  \mcB_n\}\end{math}  be any of the four sets of tableaux defined above. 

In permutation tableaux and type--B permutation tableaux, a \emph{restricted \begin{math}0\end{math}} is a \begin{math}0\end{math} which has a \begin{math}1\end{math} above it in the same column. An \emph{unrestricted row} is a row which does not contain any restricted \begin{math}0\end{math}'s (and for type--B permutation tableaux, 
also does not contain a diagonal \begin{math}0\end{math}). We let \begin{math}U_n(T)\end{math} denote the number of unrestricted rows in a tableau \begin{math}T\end{math} of size \begin{math}n\end{math}. It is also convenient to denote a topmost \begin{math}1\end{math} in a column by \begin{math}1_{T}\end{math} and a right-most restricted \begin{math}0\end{math} by \begin{math}0_{R}\end{math}. 

\emph{Corners} of a Ferrers diagram (or the associated tableau) are the cells in which both the right and bottom edges are border edges (i.e. a south step followed by a west step). In  tree--like tableaux (symmetric or not) \emph{occupied corners} are corners that contain a point. 

Our proofs will rely on techniques developed in \cite{CH} (see also \cite{HJ}). These two papers used probabilistic language and we adopt it here, too. Thus, instead of talking about the number of corners in tableaux we let \begin{math}\mbP_n\end{math}  be a probability distribution on \begin{math}\mcX_n\end{math} defined by 
  \begin{equation}\label{prob}\mbP_n(T)=\frac1{|\mcX_n|},\quad T\in\mcX_n,\end{equation}
 and we consider a random variable \begin{math}C_n\end{math} on the probability space \begin{math}(\mcX_n,\mbP_n)\end{math} defined by
 \[C_n(T)=k\quad \mbox{if and only if \begin{math}T\end{math} has \begin{math}k\end{math} corners},\quad T\in\mcX_n,\quad k\ge0.\] For convenience, let \begin{math}S_k\end{math} indicate that the \begin{math}k^{th}\end{math} step (border edge) is south and \begin{math}W_k\end{math} indicate that the \begin{math}k^{th}\end{math} step is west. Thus, 
\begin{equation}\label{eq:corners}
C_n=\sum_{k=1}^{n-1}I_{S_k,W_{k+1}},
\end{equation}
where \begin{math}I_A\end{math} is the indicator random variable of the event \begin{math}A\end{math}. 

A tableau chosen from \begin{math}\mcX_n\end{math} according to the probability measure \begin{math}\mbP_n\end{math} is usually referred to as a random tableau of size \begin{math}n\end{math} and \begin{math}C_n\end{math} is referred to as the number of corners in a random tableau of size \begin{math}n\end{math}. We let \begin{math}\mbE_n\end{math} denote the expected value with respect to the measure \begin{math}\mbP_n\end{math}. If \begin{math}c(\mcX_n)\end{math} denotes the total number of corners in tableaux in \begin{math}\mcX_n\end{math} then, in view of \eqref{prob}, we have the following simple relation:
\be\label{exp}\mbE_nC_n=\frac{c(\mcX_n)}{|\mcX_n|}\quad\mbox{or, equivalently, }\quad c(\mcX_n)=|\mcX_n|\, \mbE_nC_n.\ee

\section{Corners in Tree-Like Tableaux}\label{seccorners}
The main result of this section 
is the proof of the first conjecture of Laborde Zubieta. 
\begin{theorem} (see \cite[Conjecture~4.1]{LZ})\label{thm:conj} For \begin{math}n\ge2\end{math} we have
\[c(\mcT_n)=n!\times\frac{n+4}6.\]
\end{theorem}
To prove this, we will use the bijection between tree--like tableaux and permutation tableaux. According to Proposition~1.3 of \cite{ABN}, there exists a bijection between permutation tableaux and tree--like tableaux which transforms a tree--like tableau of shape \begin{math}F\end{math} to a permutation tableau of shape \begin{math}F'\end{math} which is obtained from \begin{math}F\end{math} by removing the SW--most edge from \begin{math}F\end{math} and  the cells of the left--most column (see Figure~\ref{TPB}). 

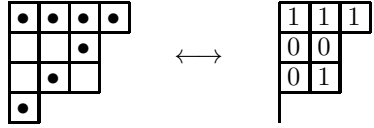
\begin{figure}[htbp]
\setlength{\unitlength}{0.4cm}
\begin{center}
\begin{picture}
(0,0)(7.5,6)\thicklines

\put(10,0){\line(0,1){4}}
\put(11,1){\line(0,1){3}}
\put(12,1){\line(0,1){3}}
\put(13,3){\line(0,1){1}}

\put(10,1){\line(1,0){2}}
\put(10,2){\line(1,0){2}}
\put(10,3){\line(1,0){3}}
\put(10,4){\line(1,0){3}}

\put(10.25, 1.25){\begin{math}0\end{math}}
\put(10.25, 2.25){\begin{math}0\end{math}}
\put(11.25, 2.25){\begin{math}0\end{math}}

\put(11.25, 1.25){\begin{math}1\end{math}}
\put(10.25, 3.25){\begin{math}1\end{math}}
\put(11.25, 3.25){\begin{math}1\end{math}}
\put(12.25, 3.25){\begin{math}1\end{math}}


\put(1,0){\line(0,1){4}}
\put(2,0){\line(0,1){4}}
\put(3,1){\line(0,1){3}}
\put(4,1){\line(0,1){3}}
\put(5,3){\line(0,1){1}}

\put(1,0){\line(1,0){1}}
\put(1,1){\line(1,0){3}}
\put(1,2){\line(1,0){3}}
\put(1,3){\line(1,0){4}}
\put(1,4){\line(1,0){4}}

\put(6.5,2){\begin{math}\longleftrightarrow\end{math}}

\put(1.25, 0.25){\begin{math}\bullet\end{math}}
\put(1.25, 3.25){\begin{math}\bullet\end{math}}
\put(2.25, 1.25){\begin{math}\bullet\end{math}}
\put(2.25, 3.25){\begin{math}\bullet\end{math}}
\put(3.25, 3.25){\begin{math}\bullet\end{math}}
\put(3.25, 2.25){\begin{math}\bullet\end{math}}
\put(4.25, 3.25){\begin{math}\bullet\end{math}}

\end{picture}

\vspace{3cm}
\caption{An example of the bijection between permutation tableaux and tree--like tableaux of size $7$.}\label{TPB}
\end{center}

\end{figure}
The number of corners in \begin{math}F\end{math} is the same as the number of corners in \begin{math}F'\end{math} if the last edge of \begin{math}F'\end{math} is horizontal and it is one more than the number of corners in \begin{math}F'\end{math} if the last edge of \begin{math}F'\end{math} is vertical. Furthermore, as is clear from a recursive construction described in \cite[Section~2]{CH}, any permutation tableau of size \begin{math}n\end{math} whose last edge is vertical is obtained as the unique extension of a permutation tableau of size \begin{math}n-1\end{math}. Therefore, there are \begin{math}(n-1)!\end{math} such   tableaux and we have a simple relation
\begin{equation}\label{corners}c(\mcT_n)=c(\mcP_n)+|\{P\in\mcP_n:\ S_n\}|=c(\mcP_n)+(n-1)!.\end{equation}
Thus, it suffices to determine the number of corners in permutation tableaux of size \begin{math}n\end{math}. Since \begin{math}|\mcP_n|=n!\end{math}, Equation~\eqref{exp} becomes
\be \label{expP} c(\mcP_n)=n!\, \mbE_nC_n.\ee
In order to determine the number of corners in permutation tableaux, we first have the following result.
\begin{theorem}\label{onecorner}
For permutation tableaux of size \begin{math}n\end{math}, the probability of having a corner with border edges \begin{math}k\end{math} and \begin{math}k+1\end{math} is given by
\beq*\label{1corner}
\mbP_n\left(I_{S_k,W_{k+1}}\right)=
\frac{n-k+1}{n}-\frac{(n-k)^2}{n(n-1)}.
\eeq*
\end{theorem}
\begin{proof}
The theorem can be proven by using techniques developed in \cite{CH}. Specifically, if \begin{math}k+1\le n-1\end{math} then  \begin{math}I_{S_k,W_{k+1}}\end{math} is a random variable on \begin{math}\mcP_{n-1}\end{math} (denoted by \begin{math}\mcT_{n-1}\end{math} in \cite{CH} and \cite{HJ}). A relationship between the measures on \begin{math}\mcP_{n}\end{math} and \begin{math}\mcP_{n-1}\end{math} was derived in \cite{CH} and is given by (see \cite[Equation~(7)]{CH} and \cite[Section~2, Equation~(2.1)]{HJ}),
\be\label{relmeas}
\mbE_n X_{n-1}=\frac{1}{n}\mbE_{n-1}(2^{U_{n-1}}X_{n-1})
\ee
where \begin{math}X_{n-1}\end{math} is any random variable defined on \begin{math}\mcP_{n-1}\end{math}.

Therefore,
\beq*
\mbP_n\left(I_{S_k,W_{k+1}}\right)&=&
\mbE_n\left(I_{S_k,W_{k+1}}\right)=\frac{1}{n}\mbE_{n-1}\left(2^{U_{n-1}}I_{S_k,W_{k+1}}\right)\\
&=&\frac{1}{n}\mbE_{n-1}\mbE\left(2^{U_{n-1}}I_{S_k,W_{k+1}}|\mcF_{n-2}\right),
\eeq*
where \begin{math}\mcF_{n-2}\end{math} is a \begin{math}\sigma\end{math}--subalgebra on \begin{math}\mcP_{n-1}\end{math} obtained by grouping into one set all tableaux in \begin{math}\mcP_{n-1}\end{math} that are obtained by extending the same tableau in \begin{math}\mcP_{n-2}\end{math} (we refer to \cite[Section~2]{HJ} for a detailed explanation).
Now, if \begin{math}k+1\leq n-2\end{math} then \begin{math}I_{S_k,W_{k+1}}\end{math} is measurable with respect to the \begin{math}\sigma\end{math}-algebra \begin{math}\mcF_{n-2}\end{math}. Thus by the properties of conditional expectation the above is:
\beq*
\mbE_n\left(I_{S_k,W_{k+1}}\right)&=&\frac{1}{n}\mbE_{n-1}I_{S_k,W_{k+1}}\mbE\left(2^{U_{n-1}}|\mcF_{n-2}\right).\\
\eeq*

By \cite[Equation~(4)]{CH}, the conditional distribution of \begin{math}U_n\end{math} given \begin{math}U_{n-1}\end{math} is given by
\[
\mathcal{L}(U_n|\mcF_{n-1})=1+\Bin(U_{n-1}),
\]
where \begin{math}\Bin(m)\end{math} denotes a binomial random variable with parameters \begin{math}m\end{math} and \begin{math}1/2\end{math}. By this result and the fact that \begin{math}\mbE a^{\Bin(m)}=\left(\frac{a+1}{2}\right)^m\end{math},
\begin{eqnarray}
\frac{1}{n}\mbE_{n-1}I_{S_k,W_{k+1}}\mbE\left(2^{U_{n-1}}|\mcF_{n-2}\right)&=&
\frac{1}{n}\mbE_{n-1}I_{S_k,W_{k+1}}\mbE\left(2^{1+\Bin(U_{n-2})}|\mcF_{n-2}\right)\nonumber\\
&=& \frac{2}{n}\mbE_{n-1}I_{S_k,W_{k+1}}\left(\frac{3}{2}\right)^{U_{n-2}}\nonumber\\
&=&\frac{2}{n(n-1)}\mbE_{n-2}I_{S_k,W_{k+1}}
3^{U_n-2}\label{binreduction}
\end{eqnarray}
where the last step follows from~\eqref{relmeas}.
Iterating \begin{math}(n-1)-(k+1)\end{math} times, we obtain
\begin{equation}\label{FC}
\frac{2\cdot3\cdot\dots\cdot(n-k-1)}{n(n-1)\cdot\dots\cdot(k+2)}\mbE_{k+1}I_{S_k,W_{k+1}}\left(n-k\right)^{U_{k+1}}.
\end{equation}
Thus, we need to compute 
\be\label{cond_exp}\mbE_{k+1}I_{S_k,W_{k+1}}(n-k)^{U_{k+1}}
\ee
for \begin{math}1\le k\le n-1\end{math} (note that \begin{math}k+1=n\end{math} gives \begin{math}\mbE_{n}I_{S_{n-1},W_n}\end{math} which is exactly the summand omitted earlier by the restriction \begin{math}k+1\le n-1\end{math}). 
 This can be computed as follows. First, by the tower property of the conditional expectation and the fact that \begin{math}S_k\end{math} is \begin{math}\mcF_k\end{math}--measurable, we obtain  
\[\mbE_{k+1}I_{S_k,W_{k+1}}(n-k)^{U_{k+1}}=\mbE_{k+1}I_{S_k}\mbE(I_{W_{k+1}}(n-k)^{U_{k+1}}|\mcF_k).
\]
And now
\[\mbE(I_{W_{k+1}}(n-k)^{U_{k+1}}|\mcF_k)=\mbE((n-k)^{U_{k+1}}|\mcF_k)-\mbE(I_{S_{k+1}}(n-k)^{U_{k+1}}|\mcF_k)
\]
because the two indicators are complementary. 
The first conditional expectation  on the right--hand side, by a computation similar to~\eqref{binreduction} (see also \cite[Equation~(2.2)]{HJ}) is 
\be\label{1st}(n-k)\mbE\left((n-k)^{U_{k+1}}|\mcF_k\right)=(n-k)\left(\frac{n-k+1}2\right)^{U_k}.\ee
To compute the second conditional expectation, note that on the set \begin{math}S_{k+1}\end{math}, \begin{math}U_{k+1}=1+U_k\end{math} so that  
\beq*
\mbE(I_{S_{k+1}}(n-k)^{U_{k+1}}|\mcF_k)&=&
(n-k)^{1+U_k}\mbE(I_{S_{k+1}}|\mcF_k)\\&=&(n-k)^{1+U_k}\mbP(I_{S_{k+1}}|\mcF_k)\\&=&(n-k)^{1+U_k}\frac1{2^{U_k}}\eeq*
where the last equation follows from the fact that for every tableau \begin{math}P\in\mcP_k\end{math} only one of its \begin{math}2^{U_k(P)}\end{math} extensions to a tableau in \begin{math}\mcP_{k+1}\end{math} has \begin{math}S_{k+1}\end{math} (see \cite{CH, HJ} for more details and further explanation). 
Combining with \eqref{1st} yields 
\[ \mbE(I_{W_{k+1}}(n-k)^{U_{k+1}}|\mcF_k)=(n-k)\left(\left(\frac{n-k+1}2\right)^{U_k}-\left(\frac{n-k}2\right)^{U_k}\right)
\]
and thus \eqref{cond_exp} equals 
\[(n-k)\mbE_{k+1}\left(I_{S_k}\left(\left(\frac{n-k+1}2\right)^{U_k}-\left(\frac{n-k}2\right)^{U_k}\right)\right).
\]
The expression inside the expectation is a random variable on \begin{math}\mcP_k\end{math} so that we can use the same argument as above (based on \cite[Equation~5]{CH} or \cite[Equation~(2.1)]{HJ}) to reduce the size by one and obtain that the expression above is 
\[\frac{n-k}{k+1}\mbE_{k}I_{S_k}\left(\left(n-k+1\right)^{U_k}-\left(n-k\right)^{U_k}\right).
\]
 Furthermore, 
on the set \begin{math}S_k\end{math}, \begin{math}U_k=U_{k-1}+1\end{math} so that the above is
\[\frac{n-k}{k+1}\mbE_{k}\left(
\left(\left(n-k+1\right)^{1+U_{k-1}}-\left(n-k\right)^{1+U_{k-1}}\right)\mbE(I_{S_k}|\mcF_{k-1})\right),
\]
which, by the same argument as above,  equals
\[\frac{n-k}{k+1}\mbE_{k}\left(
\left(\left(n-k+1\right)^{1+U_{k-1}}-\left(n-k\right)^{1+U_{k-1}}\right)\frac1{2^{U_{k-1}}}\right).
\]
After reducing the size one more time we obtain
\begin{equation}
\label{MP}\frac{n-k}{(k+1)k}\left(\mbE_{k-1}
\left(n-k+1\right)^{1+U_{k-1}}-\mbE_{k-1}\left(n-k\right)^{1+U_{k-1}}\right).
\end{equation}
As computed in \cite[Equation~(2.4)]{HJ} for a positive integer \begin{math}m\end{math} the generating function of \begin{math}U_m\end{math} is given by
\[
\mbE_m z^{U_m}=\frac{\Gamma(z+m)}{\Gamma(z)m!}.
\]
(There is an obvious omission in (2.4) there; the \begin{math}z+n\end{math} in the third expression should be \begin{math}z+n-1\end{math}.)
Using this with \begin{math}m=k-1\end{math} and \begin{math}z=n-k+1\end{math} and then with \begin{math}z=n-k\end{math} we obtain
\begin{equation}\label{FP}
\mbE_{k-1}
\left(\left(n-k+1\right)^{1+U_{k-1}}\right)=(n-k+1)\frac{(n-1)!}{(n-k)!(k-1)!}
\end{equation}
and
\be
\mbE_{k-1}\left(\left(n-k\right)^{1+U_{k-1}}\right)=(n-k)
\frac{(n-2)!}{(n-k-1)!(k-1)!}.
\label{SP}
\ee
Combining Equations (\ref{FC}), (\ref{MP}), (\ref{FP}), and (\ref{SP}),
\beq*
&&\mbE_n\left(I_{S_k,W_{k+1}}\right)=\\&&\quad\frac{(n-k-1)!(k+1)!}{n!}\cdot\frac{n-k}{k(k+1)}\left(\frac{(n-k+1)(n-1)!}{(k-1)!(n-k)!}-\frac{(n-k)(n-2)!}{(k-1)!(n-k-1)!}\right)\\&&\quad=
\frac{n-k+1}{n}-\frac{(n-k)^2}{n(n-1)},
\eeq*
and the conclusion follows.
\end{proof}

The relationship between permutation tableaux and tree--like tableaux given by \eqref{corners} allows us to deduce the following corollary to Theorem~\ref{1corner}.
\begin{corollary}\label{specificcornerT}
For tree--like tableaux of size \begin{math}n\end{math}, \begin{math}n\geq 2\end{math}, the probability of having a corner with border edges \begin{math}k\end{math} and \begin{math}k+1\end{math} is given by
\[\mbP_n\left(I_{S_k,W_{k+1}}\right)=
\begin{cases}
\frac{n-k+1}{n}-\frac{(n-k)^2}{n(n-1)}&k=1,\dots,n-1;\\
\frac1n &k=n.
\end{cases}
\]
\end{corollary}

Finally, we establish the following result which, when combined with \eqref{corners} and \eqref{expP}, completes the proof of Theorem~\ref{thm:conj}.
\begin{theorem}\label{thm:exp-c}
For  permutation tableaux of size \begin{math}n\end{math} we have
\[\mbE_n C_n=\frac{n+4}6-\frac1n.\]
\end{theorem}
\begin{proof}
In view of \eqref{eq:corners} we are interested in 
\beq*
\mbE_n\left(\sum_{k=1}^{n-1}I_{S_k,W_{k+1}}\right)=\sum_{k=1}^{n-1}\mbE_n\left(I_{S_k,W_{k+1}}\right).
\eeq*
Therefore, the result is obtained by summing the expression from Theorem~\ref{onecorner} from \begin{math}k=1\end{math} to \begin{math}n-1\end{math}. 
\end{proof}
To conclude this section, note that Theorem~\ref{thm:conj} could also be obtained by summing the expression from Corollary~\ref{specificcornerT} from \begin{math}k=1\end{math} to \begin{math}n\end{math}.

\section{Corners in Symmetric Tree-Like Tableaux}\label{sec:symtab}
The main result of this section concerns the second conjecture of Laborde Zubieta. 
\begin{theorem} (see \cite[Conjecture~4.2]{LZ}) \label{thm:conjB} For \begin{math}n\ge2\end{math} we have
\[c(\mcT^{sym}_{2n+1})=2^n\times n!\times\frac{4n+13}{12}.\]
\end{theorem}
As in Section~\ref{seccorners}, we will use a bijection between symmetric tree--like tableaux and type--B permutation tableaux to relate the corners of  \begin{math}\mcT^{sym}_{2n+1}\end{math} to the corners of \begin{math}\mcB_n\end{math}. In Section~2.2 of \cite{ABN}, it was mentioned that there exists such a bijection; however, no details were given. Thus, we give a description of one such bijection which will be useful to us (see Figure~\ref{TBPB}).  

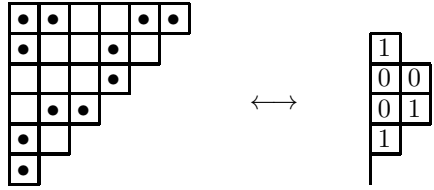
\begin{figure}[htbp] 
\setlength{\unitlength}{0.4cm}
\begin{center}
\begin{picture}
(0,0)(7,6)\thicklines

\put(12,0){\line(0,1){5}}
\put(13,1){\line(0,1){4}}
\put(14,2){\line(0,1){2}}

\put(12,1){\line(1,0){1}}
\put(12,2){\line(1,0){2}}
\put(12,3){\line(1,0){2}}
\put(12,4){\line(1,0){2}}
\put(12,5){\line(1,0){1}}

\put(12.25, 1.25){\begin{math}1\end{math}}
\put(12.25, 4.25){\begin{math}1\end{math}}
\put(13.25, 2.25){\begin{math}1\end{math}}

\put(12.25, 2.25){\begin{math}0\end{math}}
\put(12.25, 3.25){\begin{math}0\end{math}}
\put(13.25, 3.25){\begin{math}0\end{math}}

\put(0,0){\line(0,1){6}}
\put(1,0){\line(0,1){6}}
\put(2,1){\line(0,1){5}}
\put(3,2){\line(0,1){4}}
\put(4,3){\line(0,1){3}}
\put(5,4){\line(0,1){2}}
\put(6,5){\line(0,1){1}}

\put(0,6){\line(1,0){6}}
\put(0,5){\line(1,0){6}}
\put(0,4){\line(1,0){5}}
\put(0,3){\line(1,0){4}}
\put(0,2){\line(1,0){3}}
\put(0,1){\line(1,0){2}}
\put(0,0){\line(1,0){1}}

\put(8,2.5){\begin{math}\longleftrightarrow\end{math}}

\put(0.25, 0.25){\begin{math}\bullet\end{math}}
\put(0.25, 1.25){\begin{math}\bullet\end{math}}
\put(0.25, 4.25){\begin{math}\bullet\end{math}}
\put(0.25, 5.25){\begin{math}\bullet\end{math}}
\put(1.25, 2.25){\begin{math}\bullet\end{math}}
\put(1.25, 5.25){\begin{math}\bullet\end{math}}
\put(2.25, 2.25){\begin{math}\bullet\end{math}}
\put(3.25, 3.25){\begin{math}\bullet\end{math}}
\put(3.25, 4.25){\begin{math}\bullet\end{math}}
\put(4.25, 5.25){\begin{math}\bullet\end{math}}
\put(5.25, 5.25){\begin{math}\bullet\end{math}}

\end{picture}

\vspace{3cm}
\caption{An example of the bijection $F$as defined in Lemma~\ref{Bijection} between type--B permutation tableaux of size $5$ and symmetric tree--like tableaux of size $11$.}\label{TBPB}
\end{center}

\end{figure}
\begin{lemma}\label{Bijection} 
Consider \begin{math}F: \mathcal{T}_{2n+1}^{sym}\rightarrow\mathcal{B}_n\end{math} defined by the following rules,
\begin{enumerate}
\item Replace the topmost point in each column with \begin{math}1_{T}\end{math}'s. \label{TR1}
\item Replace the leftmost points in each row with \begin{math}0_{R}\end{math}'s \label{TR2}
\item Fill in the remaining cells according to the rules of type--B permutation tableaux.\label{TR3}
\item Remove the cells above the diagonal.\label{TR4}
\item Remove the first column.\label{TR5}
\end{enumerate}
and \begin{math}F^{-1}: \mathcal{B}_n\rightarrow\mathcal{T}^{sym}_{2n+1}\end{math} defined by:
\begin{enumerate}
\item Add a column and point all cells except those in a restricted row. \label{BR1}
\item Replace all \begin{math}0_{R}\end{math}'s with points unless 
that \begin{math}0_{R}\end{math} is in the same row as a diagonal \begin{math}0\end{math}.\label{BR2}
\item Replace all non-diagonal \begin{math}1_{T}\end{math}'s with points.\label{BR3}
\item Delete the remaining numbers, 
add a pointed box in the upper--left--hand corner (the root point), and then add the boxes necessary to make the tableau symmetric.\label{BR4}
\end{enumerate}
Then \begin{math}F\end{math} is a bijection between \begin{math}\mathcal{T}^{sym}_{2n+1}\end{math} and \begin{math}\mathcal{B}_n\end{math}.
\end{lemma}
\begin{proof}
The details of this proof are straightforward and will be given in the full version of this paper \cite{HL}.
\end{proof}

As mentioned earlier, Lemma~\ref{Bijection} will allow us to relate the corners of symmetric tree--like tableaux to the corners of type--B permutation tableaux.  To carry out the calculations for type--B permutation tableaux we will develop techniques similar to those developed in \cite{CH} for permutation tableaux. First, we briefly describe an extension procedure for \begin{math}B\end{math}--type tableaux that mimics a construction given in \cite[Section~2]{CH}. Fix any \begin{math}B\in\mcB_{n-1}\end{math} and let  \begin{math}U_{n-1}=U_{n-1}(B)\end{math}  be the number of unrestricted rows in \begin{math}B\end{math}. We can extend the size of \begin{math}B\end{math} to \begin{math}n\end{math} by inserting a new row or a new column. The details of this insertion will be left for the full version of this paper. However, if \begin{math}U_n\end{math} is the number of unrestricted rows in the extended tableaux, \begin{math}U_n=1,\dots,U_{n-1}+1\end{math}, the (conditional) probability  that \begin{math}U_n=U_{n-1}+1\end{math}  is given by inserting a row,
\be\label{probadd}\mbP(U_n=U_{n-1}+1|\mcF_{n-1})=\mbP(S_n|\mcF_{n-1})=\frac1{2^{U_{n-1}+1}}.\ee
(Here, analogously to permutation tableaux (see the proof  of Theorem~\ref{thm:exp-c} above or \cite[Section~2]{HJ}) \begin{math}\mcF_{n-1}\end{math} is a \begin{math}\sigma\end{math}--subalgebra on  \begin{math}\mcB_n\end{math} obtained by grouping together all tableaux in \begin{math}\mcB_n\end{math} that are obtained as the extension of the same tableau from \begin{math}\mcB_{n-1}\end{math}.) The (conditional) probability of the remaining cases is given by inserting a column,
\[
\mbP(U_n= k|\mcF_{n-1})=\frac{1}{2^{U_{n-1}+1}}\left(\binom{U_{n-1}}{k-1}+\binom{U_{n-1}}{k-1}\right)=\frac{1}{2^{U_{n-1}}}\binom{U_{n-1}}{k-1}, 
\]
for  \begin{math}k=1,\dots,U_{n-1}\end{math}. This agrees with \eqref{probadd} when \begin{math}k=U_{n+1}\end{math}. Thus, 
\[
\mathcal{L}(U_n|\mcF_{n-1})=1+\Bin(U_{n-1}),
\]
where the left--hand side means the conditional distribution of \begin{math}U_n\end{math} given \begin{math}U_{n-1}\end{math} and \begin{math}\Bin(m)\end{math} denotes a binomial random variable with parameters \begin{math}m\end{math} and \begin{math}1/2\end{math}. Note that 
this is the same relationship 
as for permutation tableaux (see \cite[Equation~(2.2)]{HJ} or \cite[Equation~4]{CH}).

 As in the case of permutation tableaux, the uniform measure \begin{math}\mbP_n\end{math} on \begin{math}\mcB_n\end{math} induces a measure (still denoted by \begin{math}\mbP_n\end{math}) on \begin{math}\mcB_{n-1}\end{math} via a mapping \begin{math}\mcB_n\to\mcB_{n-1}\end{math} that assigns to any \begin{math}B'\in \mcB_n\end{math} the unique tableau of size \begin{math}n-1\end{math} whose extension is \begin{math}B'\end{math}.
 These two measures on \begin{math}\mcB_{n-1}\end{math} are not identical, but  the relationship between them can be easily calculated (see \cite[Section~2]{CH} or \cite[Section~2]{HJ} for more details and calculations 
for  permutation tableaux).  Namely,  
\beq*
\mbP_n(B)&=&2^{U_{n-1}(B)+1}\frac{|\mcB_{n-1}|}{|\mcB_n|}\mbP_{n-1}(B),\quad B\in\mcB_{n-1}.\\
\eeq*
This relationship implies that for any random variable \begin{math}X\end{math} on \begin{math}\mcB_{n-1}\end{math},
\begin{eqnarray} 
\mbE_nX&=&\frac{2|\mcB_{n-1}|}{|\mcB_n|}\mbE_{n-1}(2^{U_{n-1}(B_{n-1})}X). \label{expBpre}
\end{eqnarray}
This allows us to provide a direct proof of the following well known fact,
\begin{proposition}\label{prop:|B_n|}
For all \begin{math}n\geq0\end{math}, \begin{math}|\mcB_n|=2^nn!\end{math}.
\end{proposition}

\begin{proof}
By considering all the extensions of a type--B permutation tableau of size \begin{math}n-1\end{math}, we have the following relationship,
\[
|\mcB_n|=\sum_{B\in\mcB_{n-1}}2^{U_{n-1}(B)+1}.
\]
Thus,
\beq*
|\mcB_n|&=&|\mcB_{n-1}|\mbE_{n-1}\left(2^{U_{n-1}+1}\right)\\
&=&2|\mcB_{n-1}|\mbE_{n-1}\mbE\left(2^{1+\Bin(U_{n-2})}|U_{n-2}\right)\\
&=&2\cdot2|\mcB_{n-1}|\mbE_{n-1}\left(\frac{3}{2}\right)^{U_{n-2}}\\
&=&2\cdot2|\mcB_{n-1}|\frac{2|\mcB_{n-2}|}{|\mcB_{n-1}|}\mbE_{n-2}\left(2^{U_{n-2}}\left(\frac32\right)^{U_{n-2}}\right)\\
&=&2^2\cdot2!\,|\mcB_{n-2}|\mbE_{n-2}3^{U_{n-2}}.
\eeq*
Iterating \begin{math}n\end{math} times,
\beq*
|\mcB_n|&=&2^3\cdot3!\, |\mcB_{n-3}|\mbE_{n-3}4^{U_{n-3}}
=2^{n-1}(n-1)!|\mcB_1|\mbE_1n^{U_1}\\
&=&2^nn!,
\eeq*
where the final equality holds because \begin{math}|\mcB_1|=2\end{math} and \begin{math}U_1\equiv 1\end{math}.
\end{proof}

Given Proposition~\ref{prop:|B_n|},  
 \eqref{expBpre} reads
\begin{equation}\label{expB}
\mbE_nX=\frac1n\mbE_{n-1}(2^{U_{n-1}(B_{n-1})}X).
\end{equation}
This is exactly the same expression as \cite[Equation~(7)]{CH} which means that the relationship between \begin{math}\mbE_n\end{math} and \begin{math}\mbE_{n-1}\end{math} is the same regardless of whether we are considering \begin{math}\mcP_n\end{math} or \begin{math}\mcB_n\end{math}. Thus, any computation for \begin{math}B\end{math}--type tableaux based on \eqref{expB} will lead to the same expression as the analogous computation for permutation tableaux based on \cite[Equation~(7)]{CH}.

Now we have the tools necessary to obtain a relationship between corners in symmetric tree--like tableaux and type--B permutation tableaux which is analogous to~\eqref{corners}.
\begin{lemma}\label{lemmaCSTT}
The number of corners in symmetric tree--like tableaux is given by, 
\begin{eqnarray}\label{cornersSym}c(\mcT^{sym}_{2n+1})
&=&2c(\mcB_n)+2^n(n-1)!+2^{n-1}n!.\end{eqnarray}
\end{lemma}
\begin{proof}
The bijection described in Lemma~\ref{Bijection} leads to the following relationship,
\be \label{rela}
c(\mcT^{sym}_{2n+1})=2c(\mcB_n)+2|\{B\in\mcB_n:\ S_n\}|+|\{B\in\mcB_n:\ W_1\}|.
\ee
The result is then obtained by the extension process described above. The details will be given in the full version of this paper \cite{HL}.
\end{proof}

It follows from Lemma~\ref{lemmaCSTT} that to prove Theorem~\ref{thm:conjB}, it suffices to determine the number of corners in type--B permutation tableaux of size \begin{math}n\end{math}. Since \begin{math}|\mcB_n|=2^nn!\end{math}, Equation~\eqref{exp} becomes 
\be \label{expCorns} c(\mcB_n)=2^nn!\, \mbE_nC_n.\ee
In order to determine the number of corners in type--B permutation tableaux, we first have the following result.

\begin{theorem}\label{1cornerB}
For type--B permutation tableaux of size \begin{math}n\end{math}, the probability of having a  corner with border edges \begin{math}k\end{math} and \begin{math}k+1\end{math} is given by
\beq*
\mbP_n\left(I_{M_k=S,M_{k+1}=W}\right)&=&
\frac{n-k+1}{2n}-\frac{(n-k)^2}{4n(n-1)}.
\eeq*
\end{theorem}
\begin{proof}
The proof is similar to the proof of Theorem~\ref{onecorner}, using the techniques developed in this section for type--B permutation tableaux. The details will be given in the full version of this paper \cite{HL}.
\end{proof}

The relationship between permutation tableaux and tree--like tableaux given by \eqref{cornersSym} allows us to deduce the following corollary to Theorem~\ref{1cornerB}.

\begin{corollary}\label{specificcornerST}
For symmetric tree--like tableaux of size \begin{math}2n+1\end{math}, \begin{math}n\geq 2\end{math}, the probability of having a corner with border edges \begin{math}k\end{math} and \begin{math}k+1\end{math} is given by
\[\mbP_n\left(I_{S_k,W_{k+1}}\right)=
\begin{cases}
\frac{1}{2n}&k=1\\
\frac{k}{2n}-\frac{(k-1)^2}{4n(n-1)}&k=2,\dots n,\\
\frac{1}{2}&k=n+1\\
\frac{2n-k+2}{2n}-\frac{(2n-k+1)^2}{4n(n-1)}&k=n+2,\dots 2n\\
\frac{1}{2n} &k=2n+1.
\end{cases}
\]

\end{corollary}
Finally, we establish the following result which, when combined with \eqref{cornersSym} and \eqref{expCorns}, completes the proof of Theorem~\ref{thm:conjB}.
\begin{theorem}\label{cornersTB}
For  type--B permutation tableaux of size \begin{math}n\end{math} we have
\[\mbE_n C_n=\frac{4n+7}{24}-\frac{1}{2n}.\]\end{theorem}
\begin{proof}
The result is obtained by summing the expression from Theorem~\ref{1cornerB} from \begin{math}k=1\end{math} to \begin{math}n-1\end{math}.
\end{proof}
To conclude this section, note that Theorem~\ref{thm:conjB} could also be obtained by summing the expression from Corollary~\ref{specificcornerST} from \begin{math}k=1\end{math} to \begin{math}2n+1\end{math}.

\bibliographystyle{abbrvnat}
\bibliography{corners}

\end{document}